\documentclass[12pt]{amsart}

\title{}
\author{}

\usepackage{fullpage, amsfonts, amsmath, amscd, amsthm,stmaryrd}
\usepackage{amssymb}
\usepackage{graphics}
\usepackage{graphicx}
\usepackage{amscd}

\newtheorem{theorem}{Theorem}[section]
\newtheorem{lemma}[theorem]{Lemma}

\newtheorem{proposition}[theorem]{Proposition}
\newtheorem*{thm*}{Theorem}
\newtheorem{rem}[theorem]{Remark}

\theoremstyle{definition}

\newcommand{\CF}{\widehat{CF}}
\newcommand{\HF}{\widehat{HF}}

\newcommand{\x}{{\bf x}}

\newcommand{\D}{\mathcal{D}}

\newcommand{\A}{\mathbb{A}}
\newcommand{\Z}{\mathbb{Z}}
\newcommand{\R}{\mathbb{R}}
\newcommand{\Q}{\mathbb{Q}}

\newcommand\goth[1]{\mathfrak{#1}}
\newcommand{\s}{\goth{s}}

\newcommand{\rank}{\text{rank}}

\newcommand{\Ker}{\mathrm{Ker}}

\newcommand{\Image}{\mathrm{Im}}

\begin{document}

\author{Zhongtao Wu}
\title{Perturbed Floer Homology of some fibered three manifolds}

\maketitle

\begin{abstract}
In this paper, we write down a special Heegaard diagram for a given product three manifold $\Sigma_g\times S^1$.
We use the diagram to compute its perturbed Heegaard Floer homology. 
\end{abstract}

\section{Introduction}

Heegaard Floer homology was introduced by Ozsv\'ath and Szab\'o in \cite{OSzAnn1},\cite{OSzAnn2}, and proved to be a powerful 3-manifold invariant.  The construction of the invariant requires an admissibility condition though, which in general is not met by those ``simplest'' Heegaard diagrams for a given 3-manifold $Y$ with $b_1(Y)\geq 1$.  A variant of the construction using Novikov ring overcomes this shortcoming, and in some sense embraces the ordinary homology as a special case.  The invariants, usually called \textit{perturbed} Heegaard Floer homology, proved to be useful in some situations.  For example, Jabuka and Mark made use of them in calculating Ozsv\'ath-Szab\'o invariants for certain closed 4-manifolds \cite{JM}.

This paper is aimed to compute the perturbed Heegaard Floer homologies for product three manifolds $\Sigma_g \times S^1$.  The result is a little bit surprising as we find that the homology groups are independent of the exact direction of perturbations. 

This paper is organized as follows: In section 2, we review the backgrounds of Novikov ring $\A$ and the perturbed Heegaard Floer homology.  Treating homology groups as $\A$-vector spaces, we prove a rank inequality and an Euler characteristic identity.  In section 3, we write down a special Heegaard diagram for $T^3$, and compute its perturbed Heegaard Floer homology.  Very similar argument can be applied to arbitrary torus bundles.  In section 4, we compute the homology for nontorsion Spin$^c$ structure of $\Sigma_g\times S^1$. 

\subsection*{Acknowledgment.} I would like to thank my advisor, Zolt\'an Szab\'o, for suggesting me the problem and having many helpful discussions at various points.  I am also grateful to Yinghua Ai, Joshua Greene and Yi Ni for conversations about this work.

\section{Preliminaries on Perturbed Heegaard Floer homology}

In Ozsv\'ath and Szab\'o \cite[section 11]{OSzAnn1}, they sketch a variant of Heegaard Floer homologies analogous to the perturbed version of Seiberg-Witten Floer homology.  For the construction, we work over the\textit{ Novikov ring} $\A$ (which is in fact a field) consisting of formal power series $\sum_{r\in \R} a_r T^r$, for which $a_r \in \Z_2$ and $\#\{a_r|a_r\neq 0 , r<N \} <\infty$ for any $N\in \R$, endowed with the multiplication law: 
$$(\sum_{r \in \R} a_r T^r) \cdot (\sum_{r \in \R} b_r T^r)= \sum _{r \in \R} (\sum_s a_sb_{r-s})T^r.$$

For a pointed Heegaard diagram $(\Sigma,\alpha, \beta,z)$ for $Y$, define the boundary map $\partial$ by
$$\partial^+[x,i]=\sum_{y} (\sum_{\{\phi\in \pi_2(x,y) | n_z(\phi)\leq i\}} \#\widehat{\mathcal{M}}(\phi) T^{\mathcal{A}(\phi)}\cdot [y,i-n_z(\phi)]),$$
where $\mathcal{A}(\phi)$ denotes the area of the domain $\mathcal{D}(\phi)$.  This construction depends on the area of each periodic domain, which can be thought of as a real two-dimensional cohomology class $\eta \in H^2(Y; \R)$.  And it is shown that the corresponding homology groups, denoted by $HF^+(Y,\s;\eta)$, are invariants of the underlying topological data only. 

It is a natural question to ask for an explicit dependence of $HF^\circ(Y;\eta)$ on $\eta$.  We are not quite achieving this yet, but our result provides a bound for the rank of $\HF(Y;\eta)$ as a vector space over $\A$.  More precisely, it is bounded by $\HF(Y;\omega)$ and $\HF(Y; \Omega)$ for two very special cohomology class $[\omega]$ and $[\Omega]$, where $[\omega]$ is a \textit{generic} class in the sense that $\omega(\mathcal{D})\neq 0$ for any integral periodic domain $\mathcal{D}$; and $\Omega$ is a \textit{trivial} class, i.e. $\Omega(\mathcal{D})=0$ for any periodic domain $\mathcal{D}$.

\begin{proposition}
\label{main prop}
\begin{enumerate}

 \item The rank of $\HF(Y;\Omega)$ over $\A$ is the same as the rank of the ordinary unperturbed $\HF(Y)$ over $\Z_2$.

\item The rank of $\HF(Y;\omega)$ over $\A$ is the same as the rank of the non-torsion part of the completely twisted $\underline{\HF}(Y;\Z_2[H^1(Y;\Z)])$ over the ring $\Z[H^1(Y;\Z)]$. 

\item In general, we have a rank inequality: $$\rank \HF(Y;\omega)\leq \rank\HF(Y;\eta) \leq \rank\HF(Y;\Omega)$$

\end{enumerate}
 
\end{proposition}

The proof is based on the following simple fact from linear algebra: 

\begin{lemma}\label{l1}
 The rank of a matrix $M$ is the largest integer $n$ such that there exists some $n\times n$ minor of $M$ with non-zero determinant. 
\end{lemma}

Note that lemma \ref{l1} provides us an algorithm to compute the rank of homology: choose a basis for the vector space $\CF$, and write the boundary map $\partial$ in a matrix form $M$.  By definition, $\HF=\frac{\Ker M}{\Image M}$ 
and $\dim(\Ker M)+\text{dim}(\Image M)=\text{dim}\CF $, 
so $$\rank\HF=\dim\Ker M-\dim \Image M=\dim \CF -2 \rank M$$

 In other words, in order to find the rank of $\HF$, it suffices to find the rank of $M$, which in turn is completely determined by the determinants of all its minors. 

Both $\A$ and $\Z_2[H^1(Y; \Z)]$ consist of formal power series as their elements - this is a special property we are going to employ in deciding if a determinant is zero.  More specifically, for a matrix $(M_{ij})=(T^{\phi_{ij}}) \in \text{Mat}(\Z_2[H^1(Y;\Z)])$, $$\det M= \sum _{\{\sigma_1,\sigma_2,\cdots, \sigma_n\}=\{1,2,\cdots,n\}} T^{\phi_{1\sigma_1}+\phi_{2\sigma_2}+\cdots+\phi_{n\sigma_n}}.$$  Being a formal sum, terms can't be added unless their exponents are equal.  Hence, $\det M=0$ iff we can pair all the  terms in the summand and cancel each other out.  More formally, we find $n!/2$ pairs, where within each pair of permutations $\sigma$ and $\rho$ we have $T^{\phi_{1\sigma_1}+\phi_{2\sigma_2}+\cdots+\phi_{n\sigma_n}}=T^{\phi_{1\rho_1}+\phi_{\rho_2}+\cdots+\phi_{n\rho_n}}$, or equivalently: 

 $$\phi_{1\sigma_1}+\phi_{2\sigma_2}+\cdots+\phi_{n\sigma_n}=\phi_{1\rho_1}+\phi_{\rho_2}+\cdots+\phi_{n\rho_n}$$

In general, entries of $M$ don't have to be monomials like $T^{\phi_{ij}}$; some entries could be like $T^{\phi_{ij}^{1}+\phi_{ij}^2+\cdots}$ and some might even vanish.  These happen when there are more than two holomorphic disks or no disk connecting two generators at all.  Nonetheless, we are still able to write the determinants as sums of the products of entries, and whether $\det=0$ or not still depends on the existence of the pairing aforementioned.  While finding exactly the pairing could be difficult, we will only apply the following simple philosophy:``the more terms in the summand are equal, the more likely the sum is zero.''  This philosophy is only valid in those fields with characteristic 2 and whose elements are formal sums.  Fortunately, that is so for $\A$ and $\Z_2[H^1(Y; \Z)]$.

\begin{proof}[ proof of Proposition \ref{main prop}]

Fix an admissible diagram for $Y$, and find all generators $x_i \in \CF(Y)$.  If the boundary map is given by 
$$\partial x_i= \sum_j(\sum_{\phi\in \pi_2(x_i,x_j)}\# \widehat{\mathcal{M}}(\phi) T^\phi x_j),$$ 
construct the corresponding matrix $(M_{ij})=(\sum_{\phi\in \pi_2(x_i,x_j)}\# \widehat{\mathcal{M}}(\phi) T^\phi)$. Since $M_{ij}\in \Z_2[H^1(Y;\Z)]$, it can be evaluated with respect to a given two form $\eta$, producing a matrix $(M_{ij}(\eta))=(\sum_{\phi\in \pi_2(x_i,x_j)}\# \widehat{\mathcal{M}}(\phi) T^{\eta(\phi) }) \in \text{Mat}(\A)$.  

Take an arbitrary $k\times k $ minor of $M$, and compute its determinant. Denote this function by $D$, then the corresponding determinant of $M(\eta)$ is $D(\eta)$.  As explained earlier, we want to find the likelihood for $D(\eta)=0$.  For each pair of terms, we want to check
 $$\eta((\phi_{1\rho_1}+\cdots+\phi_{k\rho_k})-(\phi_{1\sigma_1}+\cdots+\phi_{k\sigma_k}))=0.$$

Denote $(\phi_{1\rho_1}+\cdots+\phi_{k\rho_k})-(\phi_{1\sigma_1}+\cdots+\phi_{k\sigma_k})$ by $\phi$. There are two possibilities: either $\phi=0$ or $\phi \neq 0$.  Note that $\phi_{i\rho_i}$ (resp. $\phi_{i\sigma_i}$) is a holomorphic disk connecting $x_i$ and $x_{\rho_i}$ (resp. $\x_{\sigma_i}$), so $(\phi_{1\rho_1}+\cdots+\phi_{n\rho_n})-(\phi_{1\sigma_1}+\cdots+\phi_{n\sigma_n})$ corresponds to a periodic domain in $\pi_2(x_1,x_1)$. Hence, by assumption, $\Omega(\phi)=0$, $\omega(\phi)\neq 0$ when $\phi \neq 0$, while $\eta(\phi)$ may or may not be zero.  

In other words, when we write $D(\Omega)$ as a formal sum, all terms are identical.  For $D(\omega)$, none of them are identical unless they are identical in $D$ already in the first place.  For a general $D(\eta)$, the bigger  the kernel of $\eta$ is, the more terms in the summand are equal.  Therefore, $D(\omega)=0$ implies $D(\eta)=0$ and $D(\eta)=0$ implies $D(\Omega)=0$; but not the other way around.  Apply lemma \ref{l1}, we obtain part (3) of our proposition.   

When $\phi \neq 0$, $\omega(\phi) \neq 0$, so $D(\omega)$ equals zero iff $D$ equals zero.  This implies $\rank M = \rank M(\omega)$, proving part (2).  

Since all terms in $D(\Omega)$ are identical, we may replace all $T^\phi$ by 1, and denote the corresponding matrix by $M(0)$.  Then, $D(\Omega)=0$ iff $D(0)=0$, so $\rank M(\Omega)=\rank M(0)$.  Observe that $M(0)$ corresponds to the boundary map for the ordinary unperturbed $HF(Y)$, this proves part (1). 

\end{proof}

\begin{rem}
 
It is implied in the course of the proof that $HF(Y,\eta)$ does in fact depend only on the intersection $\Ker \eta \cap PD$ of all integral periodic domains.  This is a fact that we will repeatedly use throughout the paper.  

\end{rem}

Similar results hold for $HF^+$ in a non-torsion Spin$^c$ structure $\s$:

\begin{proposition}
\label{main prop+}
\begin{enumerate}

\item When $\s$ is a non-torsion Spin$^c$ structure, $HF^+(Y,s;\eta)$ is finitely generated, and the Euler characteristic $$ \chi(HF^+(Y,\s;\eta))=\chi(HF^+(Y,s))=\pm\tau_t(Y,\s),$$ where $\tau_t$ is Turaev's torsion function, with respect to the component $t$ of $H^2(Y;\R)-0$ containing $c_1(\s)$. 

 \item The rank of $HF^+(Y,\s ;\Omega)$ over $\A$ is equal to the rank of the ordinary unperturbed $HF^+(Y,\s)$ over $\Z_2$.

\item The rank of $HF^+(Y;\omega)$ over $\A$ is equal to the rank of the non-torsion part of the completely twisted $\underline{HF^+}(Y;\Z_2[H^1(Y;\Z)])$ over the ring $\Z[H^1(Y;\Z)]$. 

\item In general, as $\A$-vector spaces, we have the inequality $$\rank HF^+(Y,\s ;\omega)\leq \rank HF^+(Y,\s ;\eta) \leq \rank HF^+(Y,\s ;\Omega)$$

\end{enumerate}
 
\end{proposition}

\begin{proof}
 The first part is proved by a similar argument as in Ozsv\'ath and Szab\'o \cite[section 5]{OSzAnn2}.  And as soon as we know $HF^+(Y,s;\eta)$ finitely generated, the argument in the proof of Proposition \ref{main prop} can be adopted to prove the remaining parts.

\end{proof}

\section{Computations of $T^3$}

In this section, we compute the perturbed Heegaard Floer homology for $T^3$.  It was shown in Ozsv\'ath and Szab\'o \cite[section 8.4]{OSzAbsGr} that $\HF(T^3) \cong H^2(T^3; \Z) \oplus H^1(T^3; \Z)$.  By Proposition \ref{main prop}, this is equivalent to $\HF(T^3; \Omega)=\A^6$.  We aim to compute $\HF(T^3;\eta)$ for a general $\eta \in H^2(T^3; \Z)=\Z^3$.  Our result is:

\begin{theorem}\label{T^3}
 For a non-zero two form $\eta$, $\HF(T^3; \eta)=\A^2$. 
\end{theorem}

Our proof is based on certain ``special Heegaard Diagram'' first introduced in Ozsv\'ath and Szab\'o \cite {OSzCont}, in which some genus $2g+1$ Heegaard Diagrams were constructed for $\Sigma_g$ bundle over $S^1$.  In this paper, we use a slightly different presentation by drawing two standard $4g$-gons to represent left hand side and right hand side genus-$g$ surfaces respectively.  Two holes are drilled in either side to form a connected sum of a $2g+1$ Heegaard surface.

\begin{figure}

\begin{center}
\includegraphics[width=4.5in]{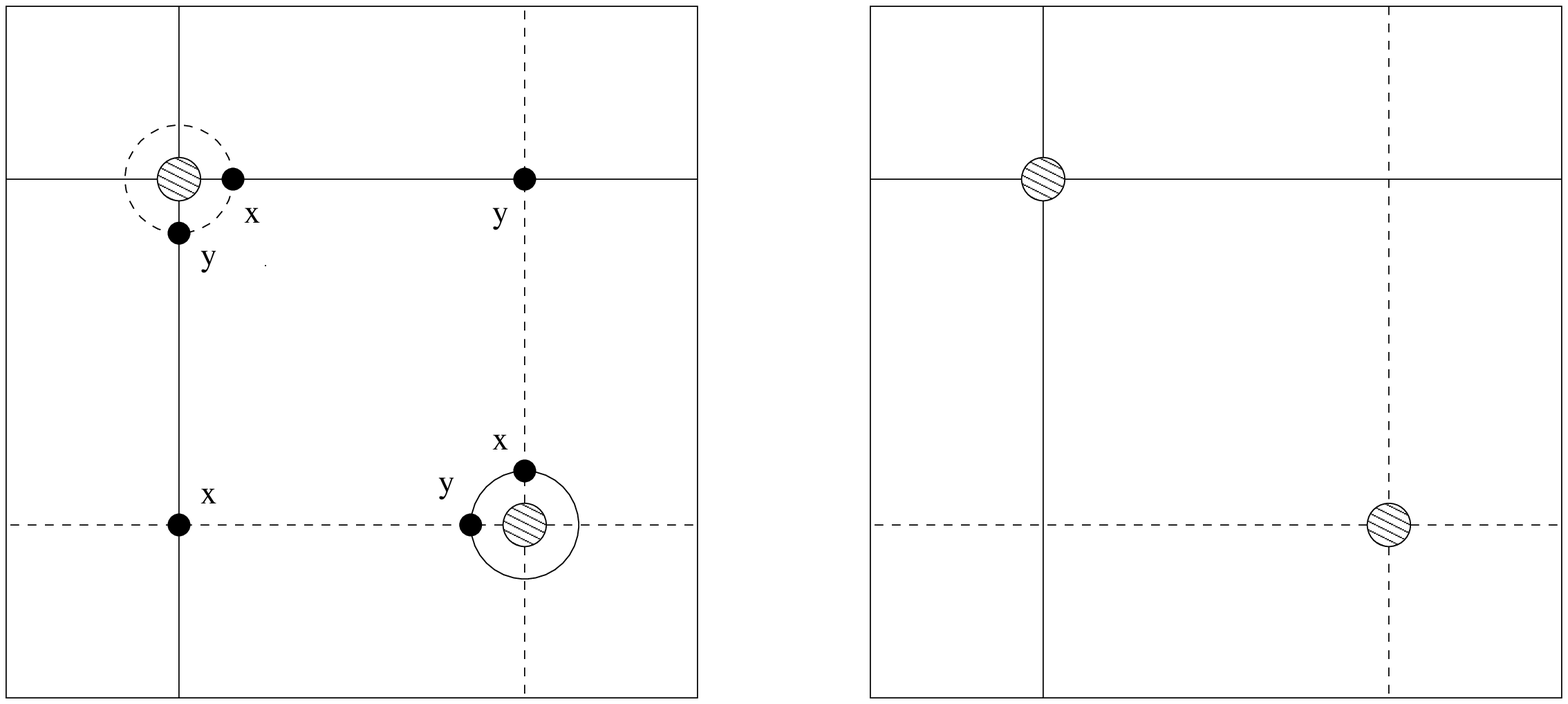}
\setlength{\unitlength}{1.02in}
\put(-1,1){$D_5$}
\put(-3.5,1){$D_1$}
\put(-2.75,1){$D_2$}
\put(-4.25,1){$D_2$}
\put(-1.75,1){$D_6$}
\put(-0.25,1){$D_6$}
\put(-1,0.25){$D_7$}
\put(-3.5,0.25){$D_3$}
\put(-2.75,0.25){$D_4$}
\put(-4.25,0.25){$D_4$}
\put(-1.75,0.25){$D_8$}
\put(-0.25,0.25){$D_8$}
\put(-1,1.75){$D_7$}
\put(-3.5,1.75){$D_3$}
\put(-2.75,1.75){$D_4$}
\put(-4.25,1.75){$D_4$}
\put(-1.75,1.75){$D_8$}
\put(-0.25,1.75){$D_8$}
\put(-1,1.375){$\alpha_2$}
\put(-3.5,1.375){$\alpha_2$}
\put(-1,0.625){$\beta_1$}
\put(-3.5,0.625){$\beta_1$}
\put(-1.375,1){$\alpha_1$}
\put(-0.625,1){$\beta_2$}
\put(-3.875,1){$\alpha_1$}
\put(-3.125,1){$\beta_2$}
\put(-3.6,1.1){$z$}
\put(-3.875,1.65){$\beta_3$}
\put(-2.8,0.625){$\alpha_3$}
\caption{\label{f1} This is a Heegaard Diagram for $T^3$: Tori are represented by rectangles with opposite sides identified, and two holes are punctured in each side, represented by shaded disks.  The Heegaard surface is divided into eight regions $D_1 \cdots D_8$ by $\alpha$'s and $\beta$'s} 
\end{center}

\end{figure}

Figure \ref{f1} is a special diagram for $T^3$: rectangles with opposite sides are identified to represent tori.  $\alpha$ and $\beta$ curves are drawn on both sides and connected through the holes to represent closed curves.  Put the base point $z$ in the region $D_1$.  Note that this is NOT an admissible diagram as periodic domains $\mathcal{D}_1:=D_2+D_4+D_6+D_8$, $\mathcal{D}_2:=D_3+D_4+D_7+D_8$ and $\mathcal{D}_3:=D_5+D_6+D_7+D_8$ have positive coefficients only.

Nevertheless, Figure \ref{f1} is useful in the computation of the perturbed Floer Homology $\HF(T^3; \eta)$; the only restriction of nonadmissibility is given by $\eta(\D_i) > 0$ for all $i$.  But at least, nonadmissible diagrams can be used to compute $\HF(T^3;\omega)$. 

\begin{lemma}
 \label{T^3 generic}
For a generic two form $\omega$, $\HF(T^3; \omega)=\A^2$.

\end{lemma}

\begin{proof}
Adjunction Inequality in \cite[section 7]{OSzAnn2} implies $\HF(T^3,\s;\omega)$ vanish for any nontorsion Spin$^c$ structures $\s$. And recall the first Chern class formula \cite[section 7.1]{OSzAnn2}:
$$\langle c_1(\s_y),[\mathcal{P}] \rangle =\chi(\mathcal{P})-2\overline{n}_z(\mathcal{P})+2\sum_{p\in y}\overline{n}_p(\mathcal{P}).$$
where $\s_y$ is a Spin$^{c}$ structure corresponding to $y$.  We find two generators $x$ and $y$ in $\CF(T^3,\mathfrak{s}_0;\omega)$, where $\s_0$ is the unique torsion Spin$^{c}$ structure of $T^3$.  

Observe that $D_1$ is a holomorphic disk connecting $x$ to $y$.  Any other holomorphic disks $\phi$ connecting $x$ to $y$ must differ $D_1$ by a periodic domain with Maslov index 0, hence $\phi$ can be written as $D_1+k_1\D_1+k_2\D_2+k_3\D_3$ for some integers $k_1$, $k_2$ and $k_3$.  A holomorphic disk has nonnegative coefficient in all regions, in particular $D_2$,$D_3$ and $D_5$.  Hence $k_i\geq 0$, which implies that $\phi$ strictly contains $D_1$.

We claim that there is no holomorphic disk connecting $y$ to $x$.  Otherwise, suppose $\psi$ is a disk connecting $y$ to $x$ with the smallest area, then 
 \begin{align*}
 (\partial^{+})^2 [x,i] &= (\partial^{+})(T^{\omega(D_1)}[y,i-1]+\cdots) \\
  &= T^{\omega(D_1)}\cdot T^{\omega(\psi)}[x,i-1-n_z(\psi)]+\text{higher order terms in $T$},
                                                                                                                                                                  \end{align*}
contradicting to $(\partial^{+})^2=0$. 

Hence, $\partial y=0$.  And for any holomorphic disk $\phi$ connecting $x$ to $y$, we have $n_z(\phi) \neq 0$. So $\partial x=0$, and consequently $\HF(T^3;\omega)=\A^2$.       

\end{proof}

Certain modifications on Figure \ref{f1} enable us to compute the perturbed Floer homology for some other two form $\eta$.  For example, Figure \ref{2} can be used for $\eta_1$ with $\eta_1(\D_1)=\eta_1(\D_2)=0$; and Figure \ref{3} can be used for $\eta_2$ with $\eta_2(\D_1)=\eta_2(\D_3)=0$.  In both cases, there are two generators $x$ and $y$, and no boundary map by a similar argument.  Hence, $\HF(Y;\eta_1)=\HF(Y;\eta_2)=\A^2$.

\begin{figure}
\begin{center}
\includegraphics[width=4.5in]{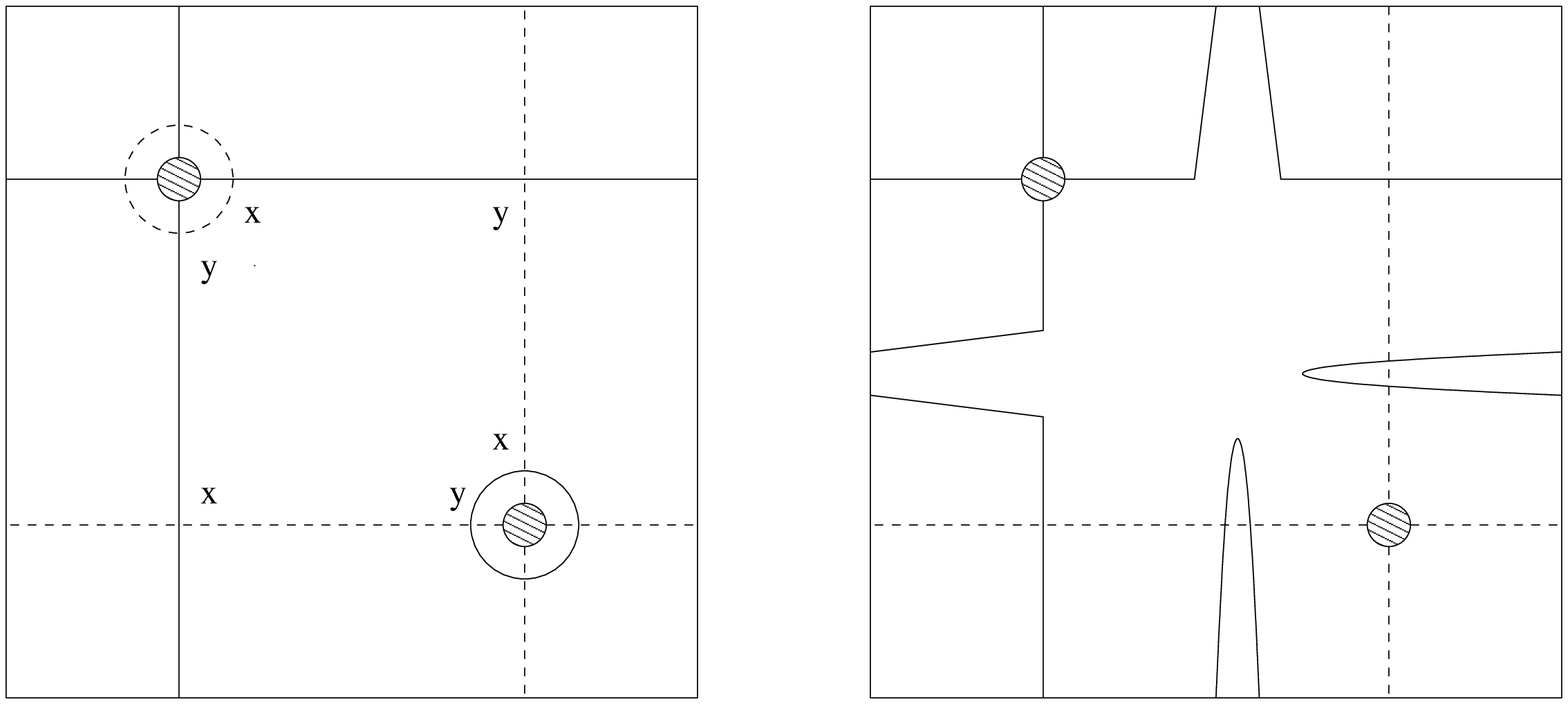}
\setlength{\unitlength}{1.02in}
\put(-1,1.375){$\alpha_2$}
\put(-3.5,1.375){$\alpha_2$}
\put(-1.1,0.625){$\beta_1$}
\put(-3.5,0.625){$\beta_1$}
\put(-1.375,1){$\alpha_1$}
\put(-0.625,1){$\beta_2$}
\put(-3.875,1){$\alpha_1$}
\put(-3.125,1){$\beta_2$}
\put(-3.6,1.1){$z$}
\put(-3.875,1.65){$\beta_3$}
\put(-2.8,0.625){$\alpha_3$}
\caption{\label{2} This is a modified Heegaard Diagram for $T^3$: $\alpha_1$ and $\alpha_2$ are twisted across $\beta_2$ and $\beta_1$ respectively.  In this diagram, there exists two form $\eta_1$ such that $\eta_1(\D_1)=\eta_1(\D_2)=0$.} 
\end{center}

\end{figure}

\begin{figure}
\begin{center}
\includegraphics[width=4.5in]{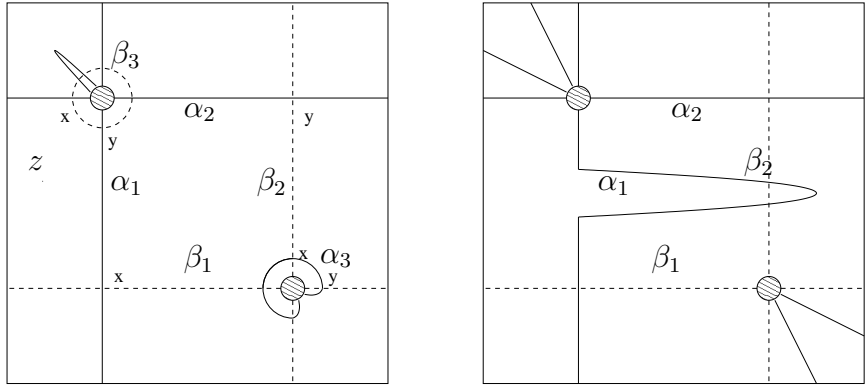}
\setlength{\unitlength}{1.02in}
\put(-1,1.375){$\alpha_2$}
\put(-3.5,1.375){$\alpha_2$}
\put(-1.1,0.6){$\beta_1$}
\put(-3.5,0.6){$\beta_1$}
\put(-1.375,1){$\alpha_1$}
\put(-0.625,1.1){$\beta_2$}
\put(-3.875,1){$\alpha_1$}
\put(-3.125,1){$\beta_2$}
\put(-4.3,1.1){$z$}
\put(-3.875,1.65){$\beta_3$}
\put(-2.8,0.625){$\alpha_3$}

\caption{ \label{3} This is a modified Heegaard Diagram for $T^3$: $\alpha_1$ is twisted across $\beta_2$, and $\alpha_3$ is winding across $\beta_3$.  In this diagram, there exists a two form $\eta_2$ such that $\eta_2(\D_1)=\eta_1(\D_3)=0$.}
\end{center}

\end{figure}
 
Figure \ref{4} is another Heegaard diagram for $T^3$, and it is admissible.  Unlike previous cases though, this time we have six generators, labeled by $x,y, p, p',q$ and $q'$, which is reasonable since $\HF(T^3)$ has rank six.  The boundary map in our case is complicated as well: Figure \ref{4} can be used for computing $\HF(T^3,\Omega)$,$\HF(T^3,\eta_1)$, $\HF(T^3,\eta_2)$ and $\HF(T^3,\omega)$, and the answers are $\A^6$ and $\A^2$ respectively.  So there must exist some cancelling pair of holomorphic disks for the area form $\Omega$ that is no longer cancellable in $\eta_1$, $\eta_2$ or $\omega$.  It would be nice if all boundary maps could be found explicitly. 

\begin{figure}
\begin{center}
\includegraphics[width=4.5in]{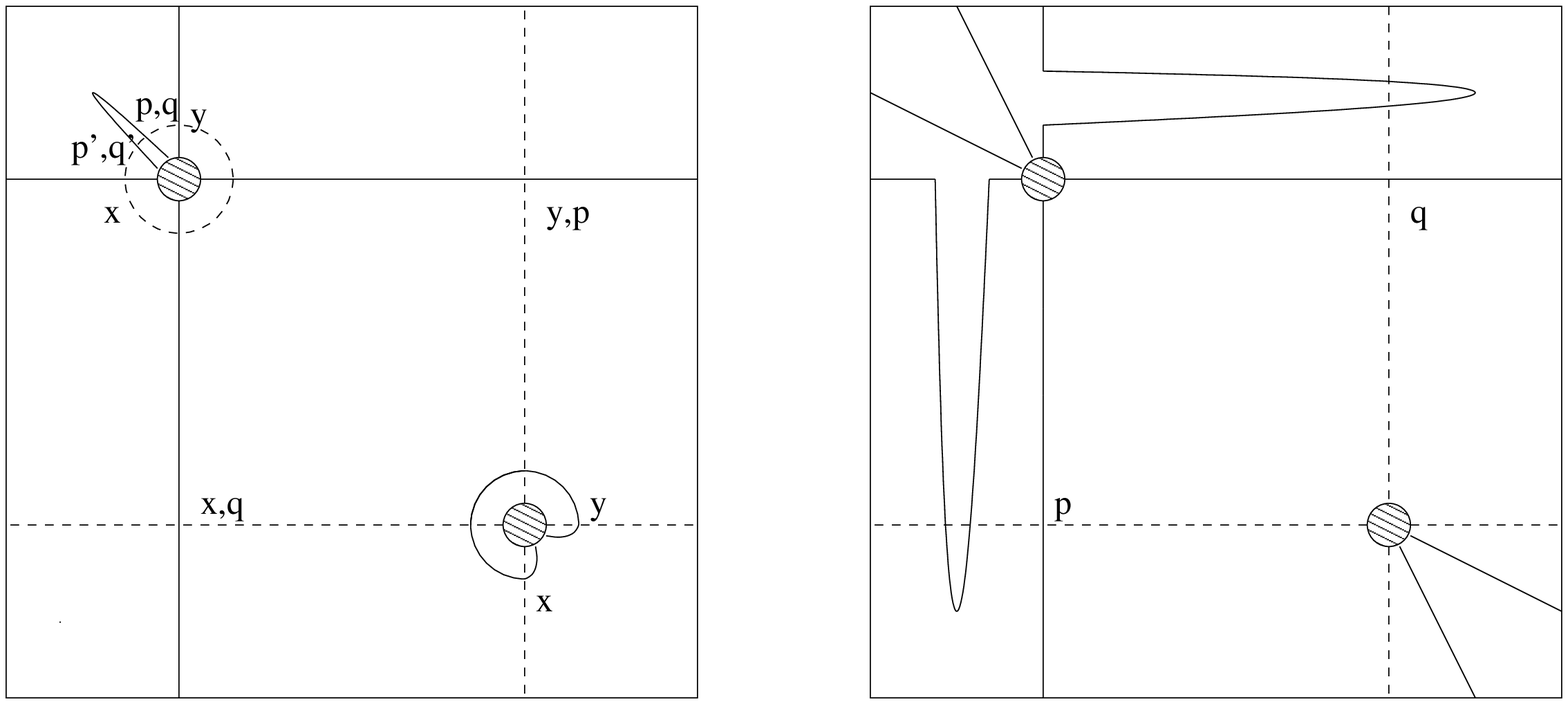}
\setlength{\unitlength}{1.02in}
\put(-1,1.375){$\alpha_2$}
\put(-3.5,1.375){$\alpha_2$}
\put(-1.1,0.625){$\beta_1$}
\put(-3.5,0.625){$\beta_1$}
\put(-1.375,1){$\alpha_1$}
\put(-0.625,1){$\beta_2$}
\put(-3.875,1){$\alpha_1$}
\put(-3.125,1){$\beta_2$}
\put(-4.3,0.2){$z$}
\put(-3.875,1.75){$\beta_3$}
\put(-2.8,0.625){$\alpha_3$}
\caption{\label{4} This is an admissible diagram for $T^3$: $\alpha_1$,$\alpha_2$ are twisted across $\beta_2$,$\beta_1$ respectively, and $\alpha_3$ is winding across $\beta_3$.  In this diagram, there exists two form $\eta$ such that $\eta(\D_1)=\eta(\D_2)=\eta(\D_3)=0$.} 
\end{center}
\end{figure}

Now, we are prepared for the proof of Theorem \ref{T^3}.  The idea is to start from some special two form $\eta'$ with the properties $\HF(T^3;\eta')=\A^2$ and $\Ker(\eta')$ a co-dimension-1 subspace of $H^2(T^3;\Q)=\Q^3$ (Both $\eta_1$ and $\eta_2$ meet the requirements).  Then, we look for some element of the large automorphism group of $T^3$ to map $\Ker(\eta')$ to some given hyperplane of $\Q^3$, namely $\Ker(\eta)$.  Functoriarily of Heegaard Floer homology implies the corresponding map from $\HF(T^3;\eta')$ to $\HF(T^3;\eta)$ is also an isomorphism, giving $\A^2$.

\begin{proof}[ proof of Theorem \ref{T^3}]
 
As mentioned earlier, both $\eta_1$ and $\eta_2$ can serve as our $\eta'$.  Instead, we describe a nonconstructive way of finding $\eta'$ that is valid in general situation.  Fix an admissible Heegaard diagram and find all generators and boundary maps.  There are only finitely many $\phi$'s in the sense of the proof of Proposition \ref{main prop}, so we can find a hyperplane $H'$ in $\Q^3$ missing all the $\phi$'s.  Let $\eta'$ evaluate zero on the hyperplane, and nonzero elsewhere.  Clearly, $\Ker(\eta')=H'$ has co-dimension 1.  And since $\eta'$ evaluates nonzero on all $\phi$'s, it essentially plays the role of a generic form $\omega$, hence by lemma \ref{T^3 generic}, $\HF(T^3;\eta')=\HF(T^3;\omega)=\A^2$.     

Suppose $\Ker(\eta)$ is another hyperplane $H$.  It is always possible to find some element of $SL_3(\Z)$ that maps $H'$ to $H$. On the other hand, any element of $SL_3(\Z)$ can be realized as the underlying $H^2(T^3;\Z)$ map induced by some $T^3$ automorphism, say $\Phi$ in this case.  Then, $\HF(T^3; \eta)=\HF(T^3; \Phi^{*}(\eta))=\HF(T^3;\eta')=\A^2$.  

\end{proof}

\begin{rem}
 
In a recent preprint by Ai and Peters \cite{AiP}, it was shown that any torus bundle $Y$ with fiber $F$ has $\HF(Y,\eta)=\A^2$ for any $\eta$ with $\eta(F)\neq 0$.  Surgery exact sequences for perturbed Floer homology were developed and applied in that paper.  Alternatively, our method of ``special Heegaard diagram'' can be applied here with ease: The left-hand rectangle is the same as that of $T^3$, and there are the same two generators with a unique smallest holomorphic disk connecting them.    

\end{rem}

\section{Computations of $ \Sigma_g \times S^1 $}

In this section, we compute the perturbed Heegaard Floer homology of $\Sigma_g \times S^1$ for $g>1$.  Our result is:

\begin{theorem}\label{big}
For a non-zero two form $\eta$, $HF^+(\Sigma_g \times S^1, k; \eta)=(\A[U]/U)^{\binom{2g-2}{d}} $, where $d=g-1-|k|$, $k\neq0$.

\end{theorem}

Here, $HF^+(\Sigma_g\times S^1,k; \eta)$ denotes the summand of $HF^+(\Sigma_g \times S^1;\eta)$ corresponding to the Spin$^c$ structure $\s$ with $ \langle c_1(\s),[\Sigma_g] \rangle =2k$ and $\langle c_1(\s),\gamma \times S^1 \rangle=0$ for all curves $\gamma \subset \Sigma_g$.

\begin{rem}
When $k\neq 0$, i.e. $\langle c_1(\s),[\Sigma_g]\rangle \neq 0$, perturbations in $\Sigma_g$-direction doesn't have any effect on the Heegaard Floer homology.  Hence, we can restrict our consideration of $\eta$ to the subspace $H^2(\Sigma_g; \Z)$ of $H^1(\Sigma_g \times S^1; \Z)$. 

\end{rem}

We can compare this result with the unperturbed case computed by Ozsv\'ath and Szab\'o in \cite[section 9]{OSzKnot}:

\begin{theorem}\label{X(g,d)}
Fix an integer $k \neq 0$.  Then, there is an identification of $\Z$-modules $$HF^+(\Sigma_g\times S^1,k) \cong X(g,d),$$ where $d=g-1-|k|$, and $$X(g,d)=\bigoplus^d_{i=0} \Lambda^{2g-i} H^1(\Sigma_g) \otimes_\Z (\Z[U]/U^{d-i+1}).$$ 

\end{theorem}

It's interesting to compare the Euler characteristic of $HF^+$.  Recall the following combinatorial identity:

\begin{lemma}
$\sum_{i=1}^m (-1)^{i+1}i\binom{2g}{m-i}=\binom{2g-2}{m-1}$ 

\end{lemma}

\begin{proof}
 
Write out the identity $\frac{x}{(1+x)^2} (1+x)^{2g}=x(1+x)^{2g-2}$ in formal series $$(\sum_{i=0}^{\infty} (-1)^{i+1}ix^i)\cdot (\sum_{i=0}^{\infty}\binom{2g}{m-i}x^{m-i})=\sum_{m=0}^{\infty}\binom{2g-2}{m-1}x^m,$$ and compare their coefficients for $x^m$.
\end{proof}

Hence, replace $d$ by $m-1$ in the formula, we have $$\chi(HF^+(\Sigma_g \times S^1,k))=\sum_{i=0}^d (-1)^{i+1}(d-i+1)\binom{2g}{i}=(-1)^{d-1}\binom{2g-2}{d}.$$  
This agrees with the Euler characteristic of $HF^+({\Sigma_g \times S^1, k; \eta})$ as expected from Proposition \ref{main prop+}.  In fact, we will use the Euler characteristic as one of the key ingredients in our proof of Theorem \ref{big}.

Just like the case of $T^3$, we divide the proof of Theorem \ref{big} into two steps:

\textit{Step 1:} We use a special Heegaard diagram for $\Sigma_g \times S^1$ in Figure \ref{Sigma2}.  There are two generators in spin$^c$ structures $k=g-1$, marked out in the figure by dots and squares.  In general, there are $2 \binom {2g-1}{d} $ generators in Spin$^c$ structure $k=g-1-d$, obtained by moving $d$ of the intersection points between $\alpha_i$ and $\beta_i$ ($i\leq 2g$) from the upper polygon to the lower polygon.  These generators are further divided into four classes: 
\begin{itemize}
 \item \textit{Class A} consists of $\binom{2g-2}{d-1}$ generators.  These generators have the intersection between $\alpha_{2g-1}$ and $\beta_{2g-1}$ in the lower polygon.

\item \textit{Class A'} consists of $\binom{2g-2}{d-1}$ generators.  These generators have the intersection between $\alpha_{2g}$ and $\beta_{2g}$ in the lower polygon.

\item \textit{Class B} consists of $\binom{2g-2}{d}$ generators.  These generators have the intersection between $\alpha_{2g}$ and $\beta_{2g}$ in the upper polygon.

\item \textit{Class B'} consists of $\binom{2g-2}{d}$ generators.  These generators have the intersection between $\alpha_{2g-1}$ and $\beta_{2g-1}$ in the upper polygon.

\end{itemize}





\begin{figure}
\begin{center}
\includegraphics[width=3in]{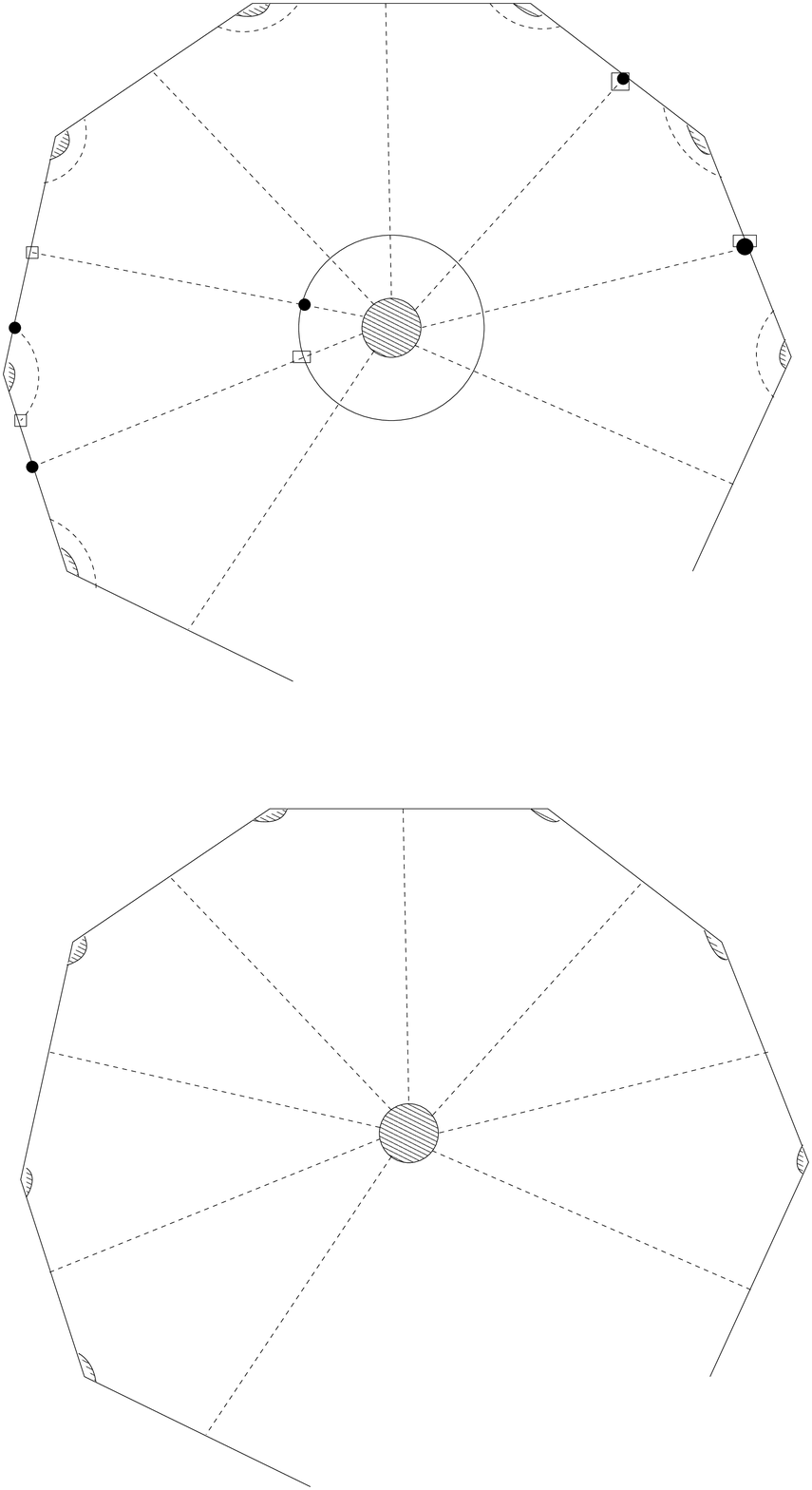}
\setlength{\unitlength}{0.0084in}
\put(-15,100){$\alpha_2$}
\put(-27,210){$\alpha_1$}
\put(-80,285){$\alpha_2$}
\put(-210,310){$\alpha_1$}
\put(-320,270){$\alpha_{2g}$}
\put(-370,210){$\alpha_{2g-1}$}
\put(-360,80){$\alpha_{2g}$}
\put(-310,20){$\alpha_{2g-1}$}
\put(-20,460){$\alpha_2$}
\put(-32,570){$\alpha_1$}
\put(-85,645){$\alpha_2$}
\put(-215,670){$\alpha_1$}
\put(-325,630){$\alpha_{2g}$}
\put(-375,570){$\alpha_{2g-1}$}
\put(-365,440){$\alpha_{2g}$}
\put(-315,380){$\alpha_{2g-1}$}
\put(-200,600){$\beta_1$}
\put(-140,580){$\beta_2$}
\put(-100,540){$\beta_1$}
\put(-105,480){$\beta_2$}
\put(-255,575){$\beta_{2g}$}
\put(-270,525){$\beta_{2g-1}$}
\put(-270,480){$\beta_{2g}$}
\put(-235,440){$\beta_{2g-1}$}
\put(-195,240){$\beta_1$}
\put(-135,220){$\beta_2$}
\put(-95,180){$\beta_1$}
\put(-100,120){$\beta_2$}
\put(-250,215){$\beta_{2g}$}
\put(-265,165){$\beta_{2g-1}$}
\put(-265,120){$\beta_{2g}$}
\put(-230,80){$\beta_{2g-1}$}
\put(-160,480){$\alpha_{2g+1}$}
\put(-50,500){$\beta_{2g+1}$}
\put(-280,500){$z$}
\caption{\label{Sigma2} This is a non-admissible Heegaard Diagram for $\Sigma_g\times S^1$.  Two holes are punctured in each $4g$-gons and connected to a genus $2g+1$ Heegaard surface.  The two generators in spin$^c$ structures $k=2g-2$ are marked out by dots and squares. In general, there are $2 \binom {2g-1}{d} $ generators in Spin$^c$ structure $k$, which are obtained by moving $d$ of the intersection points between $\alpha_i$ and $\beta_i$ from the upper polygon to the lower polygon. }
\end{center}

\end{figure}


Denote the hexagon region where we put the base point $z$ by $D$, and the corresponding hexagon region in the lower polygon by $D'$.  Pairs of generators from Class $A$ to $A'$ are connected by $D'$, while pairs of generators from Class $B$ to $B'$ are connected by $D$.  

We summarize all the information gathered so far for the chain complex $CF^+$ in Figure \ref{CF^+}.  If there were no other holomorphic disks besides $D$ and $D'$ in the diagram, then $HF^+=(\A[U]/U)^{\binom{2g-2}{d}}$.  However, with a little assumption on the two form $\omega$, we would be able to prove the fact without much knowledge of the boundary map $\partial$.

\begin{figure}
\begin{center}
\includegraphics[width=4.5in]{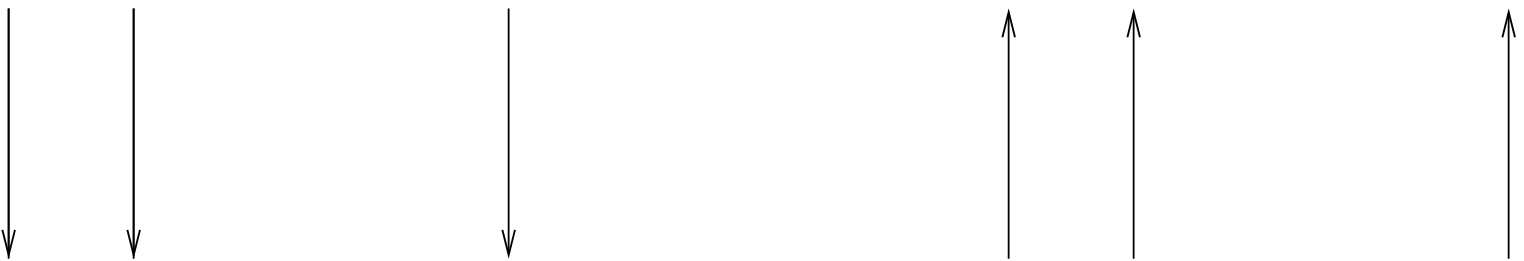}
\setlength{\unitlength}{0.0084in}
\put(0,50){$n_z=0$}
\put(-30,20){$D'$}
\put(-160,20){$D'$}
\put(-205,20){$D'$}
\put(-350,20){$D$}
\put(-480,20){$D$}
\put(-525,20){$D$}
\put(-590,50){$n_z=1$}
\put(-540,100){$b_1$}
\put(-495,100){$b_2$}
\put(-365,100){$b_{\binom{2g-2}{d}}$}
\put(-555,-20){$b_1'U^{-1}$}
\put(-495,-20){$b_2'U^{-1}$}
\put(-365,-20){$b_{\binom{2g-2}{d}}'U^{-1}$}
\put(-15,-20){$a_{\binom{2g-2}{d-1}}$}
\put(-145,-20){$a_2$}
\put(-190,-20){$a_1$}
\put(-15,100){$a_{\binom{2g-2}{d-1}}'$}
\put(-145,100){$a_2'$}
\put(-190,100){$a_1'$}
\put(-450,50){$\cdots$}
\put(-100,50){$\cdots$}
\put(-610,100){Odd}
\put(-610,-20){Even}

\caption{\label{CF^+} This diagram includes all the information we know about $CF^+$.  Class A,A',B and B' generators are denoted by $a_1,\cdots, a_{\binom{2g-2}{d-1}}$ , $a_1' \cdots, a_{\binom{2g-2}{d-1}}'$,   $b_1,\cdots, b_{\binom{2g-2}{d}}$ and $b_1' \cdots, b_{\binom{2g-2}{d}}'$ respectively.  In $\Z/2\Z$ grading,Class B,A' have odd degree, and Class A,B' have even degree.  Miraculously, these little information almost determines $HF^+$ completely. } 
\end{center}
\end{figure}

\begin{proposition}
 \label{biglemma}

For a generic two form $\omega$ with $\omega(D)=\omega(D')\ll \omega(\text{other regions})$, we have $HF^+(\Sigma_g \times S^1, k; \omega)=(\A[U]/U)^{\binom{2g-2}{d}}$, where $d=g-1-|k|$, $k\neq 0$. 
\end{proposition}

\begin{proof}
 
Use $A$, $A'$, $B$ and $B'$ to denote the vector spaces generated by Class A,A',B and B' generators respectively, and define \\
``$Odd$'':= $(B+A')\cdot (1+U^{-1}+U^{-2}+\cdots)$, \\
``$Even$'':= $(A+B')\cdot (1+U^{-1}+U^{-2}+\cdots)$,\\
``$M$'':= $B+A'\cdot(1+U^{-1}+U^{-2}+\cdots)$,\\
``$N$'':= $B\cdot(U^{-1}+U^{-2}+\cdots)$,\\
``$Ker$'':= Kernel of the boundary map $Odd\longrightarrow Even$, \\
``$Im$'':= Image of the boundary map $Even \longrightarrow Odd$, \\
``$\widetilde{Ker}$'':= projection of $Ker$ into $M$. \\
``$\widetilde{Im}$'': projection of $Im$ into $M$.

\begin{itemize}
 \item 

$Odd=M\oplus N$, $ Im \subset Ker \subset Odd$. 

\item  $Ker \cap N =0$.  

 Write elements of $N$ in the most general form $x=\sum b_iU^{-j} k_{ij} $, where $k_{ij}\in \A$.  Suppose $k_{i_1 j_1}$ is one of the coefficients with the lowest order term in $T$, then 
$$\partial x = b'_{i_1}U^{-(j_1-1)} \cdot (k_{i_1j_1}T^{\omega(D)}+\text{higher order terms in T})+\cdots .$$
But $\partial x=0$ if $x \in Ker$, which is not possible unless $x=0$.  

Hence, all information of $Ker$ is contained in $\widetilde{Ker}$, so we can restrict our attention to $\widetilde{Ker}$; same for $Im$ and $\widetilde{Im}$.

\item $\widetilde{Im} + B=M \supset \widetilde{Ker}$ 

Compute the determinant of the $\binom{2g-2}{d-1} \times \binom{2g-2}{d-1}$ $\partial$-matrix from $A$ to $A'$.  There is a unique lowest order term $T^{\binom{2g-2}{d-1}\cdot \omega(D')}$ in the determinant, hence nonzero; so the map is surjective.  Same argument carries on for larger spaces $A(1+U^{-1}+\cdots+U^{-k})$, and the map is surjective onto $A'(1+U^{-1}+\cdots+U^{-k})$.  Let $k \rightarrow \infty$, we proved $\widetilde{Im} + B=M \supset \widetilde{Ker}$.

\item
Therefore, $\rank(HF^+_{odd})\leq \rank B =\binom{2g-2}{d}$.  But $\chi(HF^+)=\binom{2g-2}{d}$, we must have $$\rank(HF^+_{odd}) =\binom{2g-2}{d}, \rank(HF^+_{even})=0.$$

\item
As shown above, we can choose a set of generators $x_1,\cdots, x_{\binom{2g-2}{d}} \in B \oplus N$ for $HF^+$. We want to prove $x_i$ in fact lies in $B$.  This would imply $x_i\cdot U =0$, finishing the proof $HF^+(\Sigma_g \times S^1, k; \omega)=(\A[U]/U)^{\binom{2g-2}{d}}$.

Up to this point, we haven't used any information of the boundary map in this special Heegaard Diagram.  Here is the place we have to use a little: upon investigating Figure \ref{Sigma2}, writing out all $k$-renormalizable periodic domain and finding out all possible topological disks with Maslov index 1, we find that there is no holomorphic disk connecting generators from Class B to B' with $n_z=0$.  In other words, the boundary map $\partial$ restricting to $B$ and $B'$ is zero.  Write $x_i=\widetilde{x_i}+y_i$, where $\widetilde{x_i}\in B$ and $y_i \in N$.  Then, $$0=\partial(x_i)=\partial(\widetilde{x_i})+\partial(y_i)=\partial(y_i)$$ But we know $Ker \cap N =0$, so $y_i=0$.  

\end{itemize}

\end{proof}

\textit{Step 2:} Since $\Sigma_g$ has a large symmetric group, the perturbed floer homology group is in some sense not sensitive to the exact direction of perturbations.  More precisely:

\begin{lemma}
\label{Sigma_g}

For any nonzero $\eta \in H^1(\Sigma_g; \Z)$, we have $HF^+(\Sigma_g \times S^1,k ; \eta)=HF^+(\Sigma_g\times S^1, k; \omega)$ and $\HF(\Sigma_g \times S^1,k ; \eta)=\HF(\Sigma_g\times S^1, k; \omega)$ as $\A$-vector spaces, for $k\neq 0$.  

\end{lemma}

\begin{proof}

The proof goes parallel to that of $T^3$: Find a special two form $\eta'$ with
$HF^\circ(\Sigma_g\times S^1; \eta')=HF^\circ(\Sigma_g \times S^1; \omega)$ and $\Ker(\eta')$ a
hyperplane $H'$ of $H^1(\Sigma_g; \Q)$.  Suppose the kernel of $\eta$ is another hyperplane $H$, it's possible to find some element in $Sp(2g;
\Z)$ that maps $H$ to $H'$.  On the other hand, a standard result in Mapping Class group implies that any element in $Sp(2g; \Z)$ is induced by some elements of the mapping class group $Mod_g$.  Functoriality of $HF^\circ$ finishes the proof.   

\end{proof}

\begin{proof}[ proof of Theorem \ref {big}]

Apply Lemma \ref{Sigma_g} and Proposition \ref{biglemma}, we have as $A$-vector space:
$$HF^+ (\Sigma_g \times S^1,k ; \eta)=\A^{\binom{2g-2}{d}},\, \HF (\Sigma_g \times S^1,k ; \eta)=\A^{2\binom{2g-2}{d}}.$$
On the other hand, as $A[U]$-module, $HF^+ (\Sigma_g \times S^1,k ; \eta)$ must have the general  $\A[U]/U^{k_1} \oplus \cdots \oplus \A[U]/U^{k_n}$.  So by consideration on rank, we must have $HF^+ (\Sigma_g \times S^1,k ; \eta)=(\A[U]/U)^{\binom{2g-2}{d}} $.

\end{proof}




\end{document}